\documentclass{amsart}

\usepackage{amssymb,graphicx,setspace,hyperref}

\newtheorem{thm}{Theorem}[section]
\newtheorem{cor}[thm]{Corollary}
\newtheorem{lem}[thm]{Lemma}
\newtheorem{prop}[thm]{Proposition}
\theoremstyle{definition}
\newtheorem{defn}[thm]{Definition}
\theoremstyle{remark}
\newtheorem{rem}[thm]{Remark}
\numberwithin{equation}{section}

\numberwithin{equation}{section}

\begin{document}

\title[Representations of Clifford Algebras]{On Representations of Clifford Algebras of Ternary Cubic Forms}


\author[Coskun]{Emre Coskun}
\address{Department of Mathematics, University of Western Ontario, London, ON, N6A 5B7 CANADA}
\email{ecoskun@uwo.ca}
\author[Kulkarni]{Rajesh S. Kulkarni}
\address{Department of Mathematics, Michigan State University, East Lansing, MI 48824}
\email{kulkarni@math.msu.edu}
\author[Mustopa]{Yusuf Mustopa}
\address{Department of Mathematics, University of Michigan, Ann Arbor, MI 48109}
\email{ymustopa@umich.edu}




\date{\today}

\begin{abstract}
In this article, we provide an overview of a one-to-one correspondence between representations of the generalized Clifford algebra $C_f$ of a ternary cubic form $f$ and certain vector bundles (called Ulrich bundles) on a cubic surface $X$. We study general properties of Ulrich bundles, and using a recent classification of Casanellas and Hartshorne, deduce the existence of irreducible representations of $C_f$ of every possible dimension.
\end{abstract}

\maketitle

\section{Introduction}

This article concerns irreducible representations of Clifford algebras, which form a natural generalization of classical Clifford algebras of quadratic forms (see Section \ref{gca} for the definition and basic properties).  These algebras are universal for linearizing forms of degree $\geq 2,$ and they have been of interest in recent years, having been studied in papers such as \cite{Cos, HT, Kul, vdB}.

The dimension of any matrix representation of a Clifford algebra associated to a nondegenerate homogeneous form of degree $d$ is necessarily divisible by $d$ (e.g. Proposition \ref{divbyd}).  The first main question is whether for a given form of degree $d$ and a positive integer $r$ there is a matrix representation of dimension $dr.$  If such representations exist, then a result of C.~Procesi (Theorem 1.8, Ch. 4 of \cite{Pro}) implies that they are parametrized by a fine moduli space.

A geometric handle on these moduli spaces may be obtained as follows.  Let $f=f(x_1, \cdots ,x_n)$ be a nondegenerate homogeneous form of degree $d,$ and let $X_f$ be the smooth hypersurface in $\mathbb{P}^{n}={\rm Proj}\;k[w,x_1, \cdots ,x_n]$ given by the equation $w^d - f(x_{1}, \cdots , x_{n})$.  A theorem of M.~Van den Bergh (\cite{vdB}) says that equivalence classes of $dr$-dimensional representations of the Clifford algebra $C_f$ of $f$ are in one-to-one correspondence with isomorphism classes of rank-$r$ vector bundles on $X_f$ whose direct images under the projection map to $\mathbb{P}^{n-1}$ (forgetting the $w$ coordinate) are the trivial vector bundle of rank $dr$ on $\mathbb{P}^{n-1}.$  Such vector bundles have been studied in a variety of algebraic and algebro-geometric contexts (e.g. \cite{BHU, BHS, CH, ESW, MP}) and they are known as \textit{Ulrich bundles.}  We show that under van den Bergh's correspondence, irreducible representations of $C_f$ correspond precisely to stable Ulrich bundles on $X_f$ (Proposition \ref{lem-irred-stable}).

So far, the work on representations of Clifford algebras has been focused on binary forms of arbitrary degree. In \cite{Kul}, it was shown that the moduli space of $d$-dimensional representations of a binary form of degree $d$ can be described as the complement of a theta divisor on the Picard variety of degree $d$ invertible sheaves of a smooth curve associated with the form. In \cite{Cos}, these results were extended to $rd$-dimensional representations for $r > 1$. In this case, the corresponding moduli space is isomorphic to an open subscheme of the (coarse) moduli space of rank $r$ semistable vector bundles on a smooth projective curve associated with the binary form.  Since the moduli spaces in the binary case are now completely understood, and the case of quadratic forms (in any number of variables) is classical, the next natural case to consider is that of ternary cubic forms.

Recently, in \cite{CH} M.~Casanellas and R.~Hartshorne completely classified families of stable Ulrich bundles on cubic surfaces, and showed that stable Ulrich bundles of rank $r$ exist for each $r \geq 1$.  Combining their results with Proposition \ref{lem-irred-stable}, we are easily able to deduce the existence of families of irreducible representations of $C_f$ having dimension $3r$ for each $r \geq 1$ (Corollary \ref{cubicreps}).

One curious fact is that the moduli space of $3$-dimensional representations is a zero-dimensional scheme supported on 72 points (see Corollary \ref{3dimreps}).  The Azumaya algebra which is the global object representing the functor is a quotient of $C_f$ in this case. It appears to be interesting to study this as well as moduli spaces of higher dimensional representations in some detail.

Concurrently with the article mentioned above, we have established strong connections between representations of Clifford algebras of cubic forms and the geometry of smooth curves representing $c_1$ of their associated Ulrich bundles.  This web of results (most of which generalize to del Pezzo surfaces that are not of the form $X_f$) is work in progress and will appear elsewhere.

\vskip0.5cm

\textbf{Acknowledgments:} The second author was partially supported by the NSF grants DMS-0603684 and DMS-1004306. The third author was supported by the NSF grant RTG DMS-0502170.

\section{Generalities on Representations of Generalized Clifford Algebras and Ulrich Bundles}
In this section, we give an overview of representations of generalized Clifford algebras and their relation to Ulrich bundles on hypersurfaces. Throughout the article, we work with an algebraically closed base field $k$ of characteristic 0.

\subsection{Generalized Clifford Algebras}
\label{gca}

Let $f(x_1, \cdots ,x_n)$ be a homogeneous form of degree $d \geq 2.$  We assume $f$ is nondegenerate, i.e. that the hypersurface $X \subseteq \mathbb{P}^{n}$ defined by $w^d=f(x_1, \cdots ,x_n)$ is nonsingular.

\begin{defn}
The \emph{generalized Clifford algebra} associated to $f,$ which we denote by $C_f,$ is defined to be the quotient of the associative $k$-algebra $k\{y_1, \cdots ,y_n\}$ by the two-sided ideal $I$ that is generated by the set
\begin{equation}
\{(\alpha_{1}y_ + \cdots +\alpha_{n}y_n)^d-f(\alpha_{1}, \cdots ,\alpha_{n}) : \alpha_{1}, \cdots ,\alpha_{n} \in k\}.
\end{equation}
\end{defn}

In the special case $d=2,$ $C_f$ is just the classical Clifford algebra associated to a nondegenerate quadratic form.

\begin{defn}
Let $C_{f}$ be the Clifford algebra associated to $f.$
\begin{itemize}
\item[(i)]{A \textit{representation of $C_{f}$} is a $k-$algebra homomorphism $\phi: C_{f} \rightarrow \textnormal{Mat}_{m}(k).$  The positive integer $m$ is the \textit{dimension} of $\phi.$}
\item[(ii)]{Two $m-$dimensional representations $\phi_{1},\phi_{2}$ of $C_{f}$ are \textit{equivalent} if there exists an invertible $\theta \in \textnormal{Mat}_{m}(k)$ such that $\phi_{1}=\theta\phi_{2}\theta^{-1}.$}
\end{itemize}
\end{defn}

Hence, an $m$-dimensional representation of $C_f$ is equivalent to matrices $A_1, \cdots ,A_n \in {\rm Mat}_{m}(k)$ for which the identity
\begin{equation}
\label{linearization}
(x_{1}A_{1} + \cdots x_{n}A_{n})^d=f(x_1, \cdots ,x_n) \cdot I_{m}
\end{equation}
holds in $\textnormal{Mat}_{m}(k[x_1, \cdots ,x_n])$.

The following result is a special case of Corollary 2 in \cite{vdB} or Proposition 1.1 in \cite{HT}.  For the reader's convenience, we provide the proof given in \cite{HT}, which is much simpler in the nondegenerate case.

\begin{prop}
\label{divbyd}
The dimension of any representation of $C_{f}$ is equal to $dr$ for some integer $r \geq 1.$
\end{prop}

\begin{proof}
Let $\phi: C_{f} \rightarrow {\rm Mat}^m(k)$ be an $m-$dimensional representation of $C_{f},$ and let $A_1, \cdots ,A_n$ be the associated $m \times m$ matrices over $k.$  Taking determinants on both sides of (\ref{linearization}), we obtain the relation
\begin{equation}
(\det(x_{1}A_{1}+ \cdots +x_{n}A_{n}))^{d}=f(x_1, \cdots ,x_n)^{m}
\end{equation}
Since $f(x_1, \cdots ,x_n)$ is nondegenerate, it is irreducible, and by unique factorization we must have that the degree-$m$ polynomial $\det(x_{1}A_{1}+ \cdots + x_{m}A_{m})$ is equal to $f(x_1, \cdots ,x_n)^r$ for some integer $r \geq 1.$  Comparing degrees shows that we must have $m=dr.$
\end{proof}

Every study of representations begins with the irreducibles, which we now define.

\begin{defn}
A representation $\phi: C_f \to {\rm Mat}_m(k)$ is called \emph{irreducible} if the image of $C_f$ generates ${\rm Mat}_m(k)$ as a $k$-algebra. Otherwise, $\phi$ is called \emph{reducible}.
\end{defn}

It is not immediate at this point that $C_f$ admits any representations, irreducible or otherwise.

Let us now describe how to associate to each $dr$-dimensional representation $\phi$ of $C_f$ a vector bundle of rank $r$ on the hypersurface $X_f$.  Define a $k$-algebra homomorphism
\begin{equation}
 \Phi: k[x_1, \cdots ,x_n, w] \rightarrow \textnormal{Mat}_{dr}(k[x_1, \cdots ,x_n, w])
\end{equation}
by setting $\Phi(x_i)=x_i \cdot I_{dr}$ for $i=1, \cdots ,n$ and $\Phi(w)=x_{1}A_{1} + \cdots + x_{n}A_{n}$. By (\ref{linearization}) this descends to a homomorphism $\overline{\Phi}: S_X \rightarrow \textnormal{Mat}_{dr}(k[x_1, \cdots ,x_n])$, where $S_X$ is the homogeneous coordinate ring of $X$. This yields an $S_X$-module structure on $k[x_1, \cdots ,x_n]^{dr}$.

Since composing $\overline{\Phi}$ with the natural inclusion $k[x_1, \cdots ,x_n] \hookrightarrow S_X$ yields the natural $k[x_1, \cdots ,x_n]$-module structure on $k[x_1, \cdots ,x_n]^{dr},$ the geometric content of our discussion may be summarized as follows: the homomorphism $\overline{\Phi}$ yields an $\mathcal{O}_X$-module $\mathcal{E}$ for which $\pi_{\ast}\mathcal{E} \cong \mathcal{O}_{\mathbb{P}^{n-1}}^{dr},$ where $\pi:X \rightarrow \mathbb{P}^{n-1}$ is the projection which forgets the variable $w$. It can be shown that the assumption that $X$ is smooth implies that $\mathcal{E}$ is locally free. The main result of \cite{vdB} (Proposition 1 in \cite{vdB}) implies that this construction yields an essentially bijective correspondence.

\begin{prop}
\label{vdbcor}
There is a one-to-one correspondence between equivalence classes of $dr$-dimensional representations of $C_f$ and isomorphism classes of vector bundles $\mathcal{E}$ of rank $r$ on the hypersurface $X_f$ such that $\pi_{\ast}\mathcal{E} \cong \mathcal{O}_{\mathbb{P}^{n-1}}^{dr}$. \hfill \qedsymbol
\end{prop}

This correspondence was used in \cite{vdB}, together with standard facts about vector bundles on curves, to show that a nondegenerate binary form of degree $d \geq 2$ admits irreducible representations of arbitrarily high dimension.  For nondegenerate forms in 3 or more variables, the study of the vector bundles on the geometric side of the correspondence is more involved.

\subsection{Ulrich bundles}
Around the same time as the appearance of \cite{vdB}, the vector bundles in Proposition \ref{vdbcor} were studied in \cite{BHU} as ``maximally generated maximal Cohen-Macaulay modules".  This study grew (partly) out of earlier work of Ulrich on Gorenstein rings.

Throughout this section, $X$ denotes a smooth hypersurface of degree $d$ in $\mathbb{P}^n.$  It should be noted that all the definitions and results up to and including Corollary \ref{assgrad} generalize naturally to smooth varieties of arbitrary codimension in $\mathbb{P}^{n}.$  While there is significant overlap between some results here and those in Section 2 of \cite{CH}, our proofs differ substantially from those in \cite{CH}.

\begin{defn}
A vector bundle $\mathcal{E}$ of rank $r$ on $X$ is \emph{Ulrich} if for some (and therefore every) linear projection $\pi:X \rightarrow \mathbb{P}^{n-1},$ we have that $\pi_{\ast}\mathcal{E} \cong \mathcal{O}_{\mathbb{P}^{n-1}}^{dr}$.
\end{defn}

As an immediate consequence of this definition and the surjectivity of the adjunction map, we have
\begin{cor}
\label{ggsect}
Any Ulrich bundle of rank $r$ on $X$ is globally generated and has $dr$ global sections. \hfill \qedsymbol
\end{cor}

A vector bundle $\mathcal{E}$ on $X$ is called \emph{arithmetically Cohen-Macaulay (ACM)} if $H^i(X,\mathcal{E}(n))=0$ for all $n \in \mathbb{Z}$ and $1 \leq i \leq \dim{X}-1.$  Using this condition,  we can characterize the Ulrich bundles on $X$ (with ${\rm dim}(X) \geq 2$) of rank $r$ as follows.
\begin{prop}
\label{ulrichacm}
Let ${\rm dim}(X) \geq 2$ and $\mathcal{E}$ be a vector bundle of rank $r$ on $X$.  Then the following are equivalent:
\begin{itemize}
\item[(i)]{$\mathcal{E}$ is Ulrich.}
\item[(ii)]{$\mathcal{E}$ is ACM and its Hilbert polynomial is $dr\binom{t+n-1}{n-1}$.}
\item[(iii)]{The $\mathcal{O}_{\mathbb{P}^{n}}$-module $\mathcal{E}$ admits a minimal graded free resolution of the form
\begin{equation}
0 \longrightarrow \mathcal{O}_{\mathbb{P}^{n}}(-1)^{dr} \longrightarrow \mathcal{O}_{\mathbb{P}^{n}}^{dr} \longrightarrow \mathcal{E} \longrightarrow 0.
\end{equation}}
\end{itemize}
\end{prop}
\begin{proof}
See Proposition 2.1 in \cite{ESW}.
\end{proof}





Next, we discuss stability of Ulrich bundles. We begin by recalling the notion of semistability (in the sense of Gieseker-Maruyama).

\begin{defn}
\label{giesstabdef}
If $\mathcal{G}$ is a torsion-free sheaf on $X$ of rank $r$, the \emph{reduced Hilbert polynomial} of $\mathcal{G}$ is $p(\mathcal{G}):=\frac{1}{r} \cdot H_{\mathcal{G}}(t)$, where $H_{\mathcal{G}}(t)$ is the Hilbert polynomial of $\mathcal{G}$.
\end{defn}

\begin{defn}
A torsion-free sheaf $\mathcal{E}$ of rank $r$ on $X$ is \emph{semistable} (resp. \emph{stable}) if for every subsheaf $\mathcal{F}$ of $\mathcal{E}$ for which $0 < \textnormal{rank}(\mathcal{F}) < r$ we have that (w.r.t. lexicographical order)
\begin{equation}
p(\mathcal{F}) \leq p(\mathcal{E}) \hskip10pt \textnormal{  (resp. }p(\mathcal{F}) < p(\mathcal{E})).
\end{equation}
\end{defn}

\begin{prop}
\label{cliffgiesstab}
Let $\mathcal{E}$ be an Ulrich bundle of rank $r \geq 1$ on $X$. Then $\mathcal{E}$ is semistable.
\end{prop}

\begin{proof}
Let $\mathcal{F}$ be a rank-$s$ torsion-free coherent subsheaf of $\mathcal{E}$.  Then $\pi_{\ast}\mathcal{F}$ is a rank-$ds$ torsion-free coherent subsheaf of $\pi_{\ast}\mathcal{E}=\mathcal{O}_{\mathbb{P}^{n-1}}^{dr},$ and since $\mathcal{O}_{\mathbb{P}^{n-1}}^{dr}$ is semistable, we have that $p(\pi_{\ast}\mathcal{F}) \leq p(\pi_{\ast}\mathcal{E})$.  Since cohomology is preserved under finite pushforward, we have that $d \cdot p(\pi_{\ast}\mathcal{F})=p(\mathcal{F})$ and $d \cdot p(\pi_{\ast}\mathcal{E})=p(\mathcal{E})$.  It follows immediately that $p(\mathcal{F}) \leq p(\mathcal{E})$.
\end{proof}

We now turn to proving the following statement, which generalizes Lemma 2 in \cite{vdB} to nondegenerate homogeneous forms in any number of variables.

\begin{prop}
\label{lem-irred-stable}
Let $f$ be a nondegenerate homogeneous form of degree $d$ in $n \geq 2$ variables.   If $\mathcal{E}$ is an Ulrich bundle on $X_f$, then the representation of the Clifford algebra $C_{f}$ associated to $\mathcal{E}$ is irreducible if and only if $\mathcal{E}$ is stable.
\end{prop}

It will be important to know that the Ulrich property is well behaved in short exact sequences (Proposition \ref{extcliff}). First, we need a lemma.

\begin{lem}
\label{locfree}
Let $g:Y \rightarrow Z$ be a finite flat surjective  morphism of smooth projective varieties, and let $\mathcal{G}$ be a coherent sheaf on $Y$ such that $g_{\ast}\mathcal{G}$ is locally free.  Then $\mathcal{G}$ is locally free.
\end{lem}
\begin{proof}
To show that $\mathcal{G}$ is locally free, we show that the stalks this sheaf are free modules over the local ring at any point. So translating the hypotheses into the local situation, we have a finite flat morphism of regular local rings $R \rightarrow S$ and a  finite $S$-module $M$ such that $M$ as an $R$-module is locally free of finite rank.  Thus $\displaystyle {\rm Ext}^i_R(M, R) = 0 \ \textnormal{for any} \  i > 0$ and $\displaystyle {\rm Hom}_R(M, R) \cong S$ as an $S$-module. We also have the change of rings spectral sequence:
\[
\displaystyle {\rm Ext}^i_S(M, {\rm Ext}^j_R(S, R)) \Rightarrow {\rm Ext}^{i + j}_R(M, R).
\]
\noindent The degeneration of this spectral sequence gives the isomorphism
\[
\displaystyle {\rm Ext}^i_S(M, S) \cong  {\rm Ext}^{i }_R(M, R) = 0
\]
\noindent for any $i >0$. So $M$ is a free $R$ module.

\end{proof}

\begin{prop}
\label{extcliff}
Consider the following short exact sequence of coherent sheaves on $X$:
\begin{equation}
0 \rightarrow \mathcal{F} \rightarrow \mathcal{E} \rightarrow \mathcal{G} \rightarrow 0
\end{equation}
If any two of $\mathcal{F}$, $\mathcal{E}$, and $\mathcal{G}$ are Ulrich bundles, then so is the third.
\end{prop}

\begin{proof}
Let $f, e$, and $e-f$ be the respective ranks of $\mathcal{F},$ $\mathcal{E},$ and $\mathcal{G}.$  Since $\pi$ is a finite morphism, we have the following exact sequence of sheaves on $\mathbb{P}^{n-1}$:
\begin{equation}
\label{cliffpush}
0 \rightarrow \pi_{\ast}\mathcal{F} \rightarrow \pi_{\ast}\mathcal{E} \rightarrow \pi_{\ast}\mathcal{G} \rightarrow 0.
\end{equation}
If $\mathcal{F}$ and $\mathcal{G}$ are Ulrich bundles, then $\pi_{\ast}\mathcal{F}$ and $\pi_{\ast}\mathcal{G}$ are trivial vector bundles on $\mathbb{P}^{n-1}.$  Therefore $\pi_{\ast}\mathcal{E}$, being an extension of trivial vector bundles on $\mathbb{P}^{n-1}$, is also trivial, so that $\mathcal{E}$ is Ulrich.

If $\mathcal{E}$ and $\mathcal{G}$ are Ulrich bundles, then $\mathcal{F}$ is locally free.  By definition $\pi_{\ast}\mathcal{E}$ and $\pi_{\ast}\mathcal{G}$ are trivial, so dualizing (\ref{cliffpush}) yields the exact sequence
\begin{equation}
0 \rightarrow \mathcal{O}_{\mathbb{P}^{n-1}}^{d(e-f)} \rightarrow \mathcal{O}_{\mathbb{P}^{n-1}}^{de} \rightarrow (\pi_{\ast}\mathcal{F})^{\vee}\rightarrow 0
\end{equation}
It follows from taking cohomology that $(\pi_{\ast}\mathcal{F})^{\vee}$ is a globally generated vector bundle of rank $df$ on $\mathbb{P}^{n-1}$ with exactly $df$ global sections, so it must be trivial.  In particular, $\pi_{\ast}\mathcal{F} \cong \mathcal{O}_{\mathbb{P}^{n-1}}^{dr}$, i.e. $\mathcal{F}$ is Ulrich.

Finally, if $\mathcal{F}$ and $\mathcal{E}$ are Ulrich bundles, then $\mathcal{G}$ is torsion-free, and arguing as before, the fact that $\pi_{\ast}\mathcal{G}$ is a globally generated torsion-free sheaf of rank $d(e-f)$ on $\mathbb{P}^{n-1}$ with exactly $d(e-f)$ global sections implies that $\pi_{\ast}\mathcal{G}$ is trivial.  Lemma \ref{locfree} then implies that $\mathcal{G}$ is locally free, hence an Ulrich bundle.
\end{proof}

\begin{lem}
\label{cliffdestab}
Let $\mathcal{E}$ be an Ulrich bundle on $X.$  Then for any Jordan-H\"{o}lder filtration
\begin{equation}
\label{jhfil}
 0=\mathcal{E}_0 \subseteq \mathcal{E}_1 \subseteq \cdots \subseteq \mathcal{E}_{m-1} \subseteq \mathcal{E}_m =
 \mathcal{E}
\end{equation}
we have that $\mathcal{E}_{i}$ is an Ulrich bundle for $1 \leq i \leq m.$

In particular, if $\mathcal{E}$ is a strictly semistable Ulrich bundle of rank $r \geq 2,$ then there exists a subbundle $\mathcal{F}$ of $\mathcal{E}$ having rank $s < r$ which is Ulrich.
\end{lem}

\begin{proof}
Fix a Jordan-H\"{o}lder filtration of $\mathcal{E}$ as in (\ref{jhfil}).  For $i=1,\cdots,m$, the sheaf $\mathcal{E}_{i}$ is torsion-free, and the quotient sheaf $\mathcal{E}_i / \mathcal{E}_{i-1}$ is both torsion-free and stable with
$p(\mathcal{E}_i / \mathcal{E}_{i-1})=p(\mathcal{E})=d \binom{t+n-1}{n-1}$.  To prove the lemma, it suffices to show that $\mathcal{E}_{1}$ is an Ulrich bundle.  Indeed, if $\mathcal{E}_{1}$ is an Ulrich bundle, then Proposition \ref{extcliff} implies that $\mathcal{E}/\mathcal{E}_{1}$ is also an Ulrich bundle, and the desired result follows by induction on the rank of $\mathcal{E}.$

Since $p(\pi_*
\mathcal{E}_1)=\frac{p(\mathcal{E}_1)}{d}=\binom{t+n-1}{n-1}$ and $p(\pi_*
\mathcal{E})=\frac{p(\mathcal{E})}{d}=\binom{t+n-1}{n-1}$, we have that $\pi_* \mathcal{E}_1$ is a destabilizing
subsheaf of $\pi_* \mathcal{E} \cong \mathcal{O}_{\mathbb{P}^{n-1}}^{dr}$ having the same reduced Hilbert polynomial as $\pi_{\ast}\mathcal{E}.$  It then follows from the semistability of $\mathcal{O}_{\mathbb{P}^{n-1}}^{dr}$ that $\pi_* \mathcal{E}_1$ is a
semistable sheaf on $\mathbb{P}^{n-1},$ so it admits a Jordan-H\"{o}lder filtration
\begin{equation}
\label{jh-g}
 0=\mathcal{G}_0 \subseteq \mathcal{G}_1 \subseteq \cdots \subseteq \mathcal{G}_{s-1} \subseteq \mathcal{G}_s = \pi_*
 \mathcal{E}_1.
\end{equation}
Now consider the short exact sequence
\begin{equation*}
 0 \to \mathcal{E}_1 \to \mathcal{E} \to \mathcal{E}/\mathcal{E}_1 \to 0,
\end{equation*}
and its pushforward
\begin{equation}\label{eq-jh}
 0 \to \pi_* \mathcal{E}_1 \to \pi_* \mathcal{E} \to \pi_* \mathcal{E}/ \pi_* \mathcal{E}_1 \to 0.
\end{equation}
Since $\pi_* \mathcal{E}$ and $\pi_* \mathcal{E}/ \pi_* \mathcal{E}_1$ have the same reduced Hilbert polynomial,
namely $\binom{t+n-1}{n-1}$, and $\pi_* \mathcal{E}$ is semistable, we have that $\pi_* \mathcal{E}/ \pi_*
\mathcal{E}_1$ is also semistable. Hence we may concatenate (\ref{jh-g}) with the Jordan-H\"{o}lder filtration of $\pi_* \mathcal{E}/ \pi_*
\mathcal{E}_1$ to obtain a Jordan-H\"{o}lder filtration of $\pi_*
\mathcal{E} \cong \mathcal{O}_{\mathbb{P}^{n-1}}^{dr}$. But the trivial filtration $0 \subseteq
\mathcal{O}_{\mathbb{P}^{n-1}} \subseteq \mathcal{O}_{\mathbb{P}^{n-1}}^{2} \subseteq \cdots \subseteq
\mathcal{O}_{\mathbb{P}^{n-1}}^{dr}$ is also Jordan-H\"{o}lder, so we have that for $j=1,\cdots,s$ the successive
quotients $\mathcal{G}_j / \mathcal{G}_{j-1}$ of the filtration (\ref{jh-g}) are isomorphic to $\mathcal{O}_{\mathbb{P}^{n-1}}$.

In particular, $\mathcal{G}_{1} \cong \mathcal{O}_{\mathbb{P}^{n-1}}$ and $\mathcal{G}_{2}/\mathcal{G}_{1} \cong
\mathcal{O}_{\mathbb{P}^{n-1}}.$  Therefore $\mathcal{G}_{2}$ is an extension of $\mathcal{O}_{\mathbb{P}^{n-1}}$ by
$\mathcal{O}_{\mathbb{P}^{n-1}}$, and since
$\textnormal{Ext}^{1}(\mathcal{O}_{\mathbb{P}^{n-1}},\mathcal{O}_{\mathbb{P}^{n-1}})=0$, this extension is split, so
that $\mathcal{G}_{2} \cong \mathcal{O}_{\mathbb{P}^{n-1}}^{2}.$  Continuing in this fashion, we see that
$\mathcal{G}_{s}=\pi_{\ast}\mathcal{E}_{1}$ is an extension of a trivial bundle by
$\mathcal{O}_{\mathbb{P}^{n-1}},$ so that $\pi_{\ast}\mathcal{E}_{1}$ is trivial.  By Lemma \ref{locfree}, $\mathcal{E}_1$
is locally free, so we may conclude that $\mathcal{E}_{1}$ is an Ulrich bundle.
\end{proof}

Combining this result with Proposition \ref{extcliff} yields the following

\begin{cor}
\label{assgrad}
If $\mathcal{E}$ is an Ulrich bundle on $X$, then the associated graded bundle of any Jordan-H\"{o}lder filtration of $\mathcal{E}$ is a direct sum of stable Ulrich bundles on $X$. \hfill \qedsymbol
\end{cor}

\textit{Proof of Proposition \ref{lem-irred-stable}:}  Let $\mathcal{E}$ be an Ulrich bundle corresponding to a reducible representation of $C_{f}$. We claim that $\mathcal{E}$ is strictly semi-stable. Indeed, reducibility implies that we may choose a proper, nontrivial subbundle $\mathcal{F}$ of $\mathcal{E}$ corresponding to a proper, nontrivial subrepresentation of $C_{f}.$  Since $\mathcal{F}$ is an Ulrich bundle, the reduced Hilbert polynomials $p(\mathcal{E})$ and $p(\mathcal{F})$ are equal.

Conversely, let $\mathcal{E}$ be a strictly semistable Ulrich bundle of rank $r$, and let $\phi_{\mathcal{E}}:C_{f} \rightarrow \textnormal{Mat}_{dr}(k)$ be the corresponding representation.  By Lemma \ref{cliffdestab} we have that $\mathcal{E}$ admits a destabilizing subbundle $\mathcal{F}$ of rank $s < r$ which is Ulrich; this corresponds to a subrepresentation $\phi_{\mathcal{F}}:C_{f} \rightarrow \textnormal{Mat}_{ds}(k)$ through which $\phi_{\mathcal{E}}$ factors.  It then follows that $\phi_{\mathcal{E}}(C_{f})$ cannot generate all of $\textnormal{Mat}_{dr}(k),$ so $\phi_{\mathcal{E}}$ is reducible. \hfill \qedsymbol

\begin{prop}
\label{irrexist}
If $C_f$ admits a representation, then it admits an irreducible representation.
\end{prop}

\begin{proof}
Let $\phi:C_{f} \rightarrow {\rm Mat}_{dr}(k)$ be a representation corresponding to an Ulrich bundle $\mathcal{E}$ of rank $r$ on $X.$  If $\mathcal{E}$ is stable, then we are done by Proposition \ref{lem-irred-stable}.  If $\mathcal{E}$ is strictly semistable, then we may look to the stable Ulrich subbundle $\mathcal{F}$ guaranteed by Lemma \ref{cliffdestab}; another application of Proposition \ref{lem-irred-stable} shows that the representation corresponding to $\mathcal{F}$ is irreducible.
\end{proof}

We conclude this section with a major existence result, which is the Theorem stated after ``STOP PRESS" in \cite{BHS}.  It should be noted that even though most hypersurfaces of degree $d \geq 3$ are not of the form $X_f,$ the proof makes serious use of generalized Clifford algebras.

\begin{thm}
\label{ulrichhyp}
(Backelin-Herzog-Sanders)  Every smooth hypersurface $X$ admits an Ulrich bundle.  \hfill \qedsymbol
\end{thm}

Combining this with Proposition \ref{irrexist}, we have

\begin{cor}
If $f=f(x_1, \cdots ,x_n)$ is a nondegenerate homogeneous form, then the generalized Clifford algebra $C_f$ admits an irreducible representation. \hfill \qedsymbol
\end{cor}

This does not give a good bound for the dimension of an irreducible representation of $C_f,$ since the rank of the Ulrich bundles guaranteed by \ref{ulrichhyp} is exponential in the number of monomials required to express the hypersurface $X.$

\section{The Case of Ternary Cubic Forms}

In the section, $X$ denotes a smooth cubic surface in $\mathbb{P}^3$ and $f=f(x_1,x_2,x_3)$ denotes a nondegenerate ternary cubic form.

We begin with a characterization of Ulrich line bundles on cubic surfaces.

\begin{prop}
\label{twisted}
Let $\mathcal{L}$ be a line bundle on $X$. Then $\mathcal{L}$ is an Ulrich line bundle if and only if $\mathcal{L} \cong \mathcal{O}_{X}(T)$, where $T$ is the class of a twisted cubic.
\end{prop}
\begin{proof}
($\Rightarrow$) Since $\mathcal{L}$ is globally generated, there exists a smooth curve $T \in |\mathcal{L}|$. We have from the adjunction formula that $c_{1}(\mathcal{L})^{2}=2g(T)+1$. Since $h^{0}(\mathcal{L})=3$, we have from Riemann-Roch that $c_{1}(\mathcal{L})^{2}=1$. Applying the adjunction formula to $T$ then shows that $g(T)=0$, so $T$ is necessarily a twisted cubic.

($\Leftarrow$) Let $T$ be a twisted cubic on $X$. We need to show that $\mathcal{O}_{X}(T)$ is ACM and has Hilbert polynomial $3\binom{t+2}{2}$. To show that $\mathcal{O}_{X}(T)$ is ACM, it suffices by Serre duality to show that $\mathcal{O}_{X}(-T)$ is ACM. If $\mathcal{I}_{T|\mathbb{P}^{3}}$ is the ideal sheaf of $T$ in $\mathbb{P}^{3},$ we have the exact sequences
\begin{equation}
0 \rightarrow \mathcal{O}_{\mathbb{P}^{3}}(-3) \rightarrow \mathcal{I}_{T|\mathbb{P}^{3}} \rightarrow \mathcal{O}_{X}(-T) \rightarrow 0
\end{equation}
\begin{equation}
\label{cubres}
0 \rightarrow \mathcal{O}_{\mathbb{P}^{3}}(-3)^{2} \rightarrow \mathcal{O}_{\mathbb{P}^{3}}(-2)^{3} \rightarrow \mathcal{I}_{T|\mathbb{P}^{3}} \rightarrow 0
\end{equation}
(The second of these is the well-known Eagon-Northcott resolution of $\mathcal{I}_{T|\mathbb{P}^{3}}.$)  Twisting both by $t \in \mathbb{Z}$ and taking cohomology, we have that $H^{1}(\mathcal{I}_{T|\mathbb{P}^{3}}(t)) \cong H^{1}(\mathcal{O}_{X}(-T+tH))$ and $h^{1}(\mathcal{I}_{T|\mathbb{P}^{3}}(t))=0$ for all $t \in \mathbb{Z}$.  Therefore both $\mathcal{O}_{X}(-T)$ and $\mathcal{O}_{X}(T)$ are ACM.

The fact that $\mathcal{O}_{X}(T)$ has Hilbert polynomial $3\binom{t+2}{2}$ follows immediately from computing $h^{0}(\mathcal{O}_{X}(T+tH))$ for $t >> 0$ using Riemann-Roch for surfaces.

\end{proof}

Since it is well-known (e.g. Section 5.4, \cite{Har}) that there are 72 linear equivalence classes of twisted cubics on $X$, we have 72 isomorphism classes of Ulrich line bundles on $X,$ so the following is immediate from Proposition \ref{vdbcor}.

\begin{cor}
\label{3dimreps}
There are exactly 72 equivalence classes of irreducible 3-dimensional representations of $X_f.$ \hfill \qedsymbol
\end{cor}

We now come to the classification of Ulrich bundles on $X$ of rank $r \geq 2$.  The following results are Theorem 3.9 and Theorem 1.1 in \cite{CH}, respectively.
\begin{thm}\label{thm-sumofcubics}
Let D be a divisor on X and let $r \geq 2$ be an integer. Then the following are equivalent:
\begin{itemize}
\item[(i)]{$D$ is linearly equivalent to a sum of twisted cubic curves $\sum_{i=1}^{r}T_i$.}
\item[(ii)]{There exists an Ulrich bundle $\mathcal{E}$ of rank $r$ with first Chern class equal to $D$.}
\item[(iii)]{$\textnormal{deg }D=3r$ and $0 \leq D.L \leq 2r$ for all lines $L$ in $X$. \hfill \qedsymbol}
\end{itemize}
\end{thm}

\begin{rem}
Proving the implication $(i) \Rightarrow (ii)$ is easy: if $D=\sum_{i=1}^{r} T_i$, where each $T_i$ is the class of a twisted cubic, one can take $\mathcal{E}=\oplus_{i=1}^{r} \mathcal{O}_{X}(T_i)$, which is Ulrich by Propositions \ref{extcliff} and \ref{twisted}.
\end{rem}

\begin{thm}\label{thm-exist-stable}
Let $D$ be a divisor on a nonsingular cubic surface $X \subseteq \mathbb{P}^3$, and let $r \geq 2$ be an integer. Then there exist stable Ulrich bundles $\mathcal{E}$ of rank $r$ on $X$ with $c_1(\mathcal{E}) = D$ if and only if $0 \leq D \cdot L \leq 2r$ for all lines $L$ on $X$, and $D.T \geq 2r$ for all twisted cubic curves $T$ on $X$, with one exception.

Moreover, if $D$ satisfies the conditions above, the moduli space $M^{s}_{X}(r; c_1, c_2)$ of stable vector bundles on $X$ of rank $r$, $c_1 = D$ and $c_2 = \frac{D^2-r}{2}$, is smooth and irreducible of dimension $D^2-2r^2+1$ and consists entirely of stable Ulrich vector bundles. \hfill \qedsymbol
\end{thm}

To see that the conditions of Theorem \ref{thm-exist-stable} are not vacuous, note that each $r \geq 1$ they are satisfied by the divisor $rH.$

\begin{cor}
\label{cubicreps}
For all $r \geq 1$ there exists a $3r-$dimensional irreducible representation of $C_f.$  Moreover, for each $r$ there are finitely many families of such representations.\hfill \qedsymbol
\end{cor}

\begin{proof}
The first statement follows immediately from setting $D=rH$ in Theorem \ref{thm-exist-stable}.  For the second statement, note that each family of stable Ulrich bundles is completely determined by the first Chern class $D,$ and that $D$ is a sum of (not necessarily distinct) twisted cubics by Theorem \ref{thm-sumofcubics}.  There are 72 classes of twisted cubics and each twisted cubic can appear in $D$ at most $r$ times, so the statement follows.
\end{proof}
Determining which family a given irreducible representation belongs to is an interesting problem that will be explored in future work.



\begin{thebibliography}{10}

\begin{singlespace}
\bibitem[BHS]{BHS}
J. Backelin, J. Herzog and H. Sanders, {\sl Matrix factorizations of homogeneous polynomials,} in
{\sl Algebra � some current trends,} L.L. Avramov and K.B. Tchakeriam, eds., Lecture Notes in
Mathematics 1352, Springer, Berlin 1988
\end{singlespace}

\begin{singlespace}
\bibitem[BHU]{BHU}
J. Brennan, J. Herzog, and B. Ulrich, {\sl Maximally generated Cohen-Macaulay modules}, Math. Scand. \textbf{61} (1987), no. 2, p. 181�203.
\end{singlespace}

\begin{singlespace}
\bibitem[CH]{CH}
M. Casanellas and R. Hartshorne, {\sl Stable Ulrich Bundles}, \href{http://arxiv.org/abs/1102.0878}{arXiv:1102.0878}.
\end{singlespace}

\begin{singlespace}
\bibitem[Cos]{Cos}
E. Coskun, {\sl The Fine Moduli Space of Representations of Clifford Algebras}, to appear in Int. Math. Res. Notices
\end{singlespace}

\begin{singlespace}
\bibitem[ESW]{ESW}
D. Eisenbud, F.-O. Schreyer, and J. Weyman, {\sl Resultants and Chow forms via exterior syzygies}, J. Amer. Math. Soc. \textbf{16} (2003), no. 3, 537-579
\end{singlespace}

\begin{singlespace}
\bibitem[Fae]{Fae}
D. Faenzi, {\sl Rank 2 arithmetically Cohen-Macaulay bundles on a nonsingular cubic surface}, J. Algebra \textbf{319} (2008), 143-186
\end{singlespace}

\begin{singlespace}
\bibitem[Har]{Har}
R. Hartshorne, {\sl Algebraic Geometry}, Springer-Verlag, 1977
\end{singlespace}

\begin{singlespace}
\bibitem[HT]{HT}
D. Haile and S. Tesser, {\sl On Azumaya algebras arising from Clifford algebras}, J. Algebra \textbf{116} (1988), no. 2, 372-384
\end{singlespace}

\begin{singlespace}
\bibitem[Kul]{Kul}
R. S. Kulkarni, {\sl On the Clifford algebra of a binary form}, Trans. Amer. Math. Soc. \textbf{355} (2003), no. 8, 3181-3208
\end{singlespace}

\begin{singlespace}
\bibitem[MP]{MP}
R. Mir\'{o}-Roig and J. Pons-Llopis, {\sl $N-$dimensional Fano Varieties of Wild Representation Type}, preprint, 2010
\end{singlespace}

\begin{singlespace}
\bibitem[Pro]{Pro}
C. Procesi, {\sl Rings with Polynomial Identities,} Marcel Dekker, New York, (1973)
\end{singlespace}

\begin{singlespace}
\bibitem[vdB]{vdB}
M. van den Bergh, {\sl Linearisations of Binary and Ternary Forms}, J. Algebra \textbf{109} (1987), p. 172-183
\end{singlespace}

\end{thebibliography}
\end{document}